\documentclass[11pt,twoside]{amsart}

\usepackage{amscd}
\usepackage{pb-diagram}
\usepackage{amssymb}
\usepackage{amsfonts}
\usepackage{latexsym}
\usepackage{verbatim}
\usepackage{amsthm}
\usepackage{amscd}
\usepackage{color}
\usepackage{hyperref}
\usepackage{enumerate}
\usepackage{xcolor}
\usepackage{pifont}
\usepackage{adjustbox}
\usepackage{mathrsfs,mathtools}
\usepackage{stmaryrd}

\usepackage{soul}

\makeatletter
\@namedef{subjclassname@2020}{%
  \textup{2020} Mathematics Subject Classification}
\makeatother

\numberwithin{equation}{section}

\newtheorem{theorem}{Theorem}[section]

\newtheorem{prop}[theorem]{Proposition}
\newtheorem{lemma}[theorem]{Lemma}
\newtheorem*{theorem*}{Theorem}

\theoremstyle{definition}

\newtheorem{rem}[theorem]{Remark}

\newcommand{\w}[1]{\wedge}

\newcommand{\e}{\begin{equation}}
\newcommand{\ee}{\end{equation}}

\newcommand\R{{\mathbb R}}

\newcommand{\Rc}{\mathrm{Rc}}

\newcommand{\sst}{\scriptscriptstyle}

\newcommand{\dez}{\left.\tfrac{d}{d\varepsilon}\right|_{\varepsilon=0}}

\newcommand\f{{\varphi}}

\title{On the dynamical behaviour of the generalized Ricci flow}

\author{Alberto Raffero}
\address{Dipartimento di Matematica ``G. Peano'' \\ Universit\`a degli Studi di Torino\\ Via Carlo Alberto 10\\10123 Torino\\ Italy}
\email{alberto.raffero@unito.it}

\author{Luigi Vezzoni}
\address{Dipartimento di Matematica ``G. Peano'' \\ Universit\`a degli Studi di Torino\\ Via Carlo Alberto 10\\10123 Torino\\ Italy}
\email{luigi.vezzoni@unito.it}


\begin{document}
\begin{abstract}
Motivated by M\"uller-Haslhofer results on the dynamical stability and instability of Ricci-flat metrics under the Ricci flow, 
we obtain dynamical stability and instability results for pairs of Ricci-flat metrics and vanishing 3-forms under the generalized Ricci flow.   
\end{abstract}
\maketitle

%%%%%%%%%%%%%%%%%%%%%%%%%%%%%%%%%%%%%%%%%%%%%%%%%%%%%%%%%%%%%%%%%%%%%%%%%%%%%%%%%%%%%%%%%
%%%%%%%%%%%%%%%%%%%%%%%%%%%%%%%%%%%%%%%%%%%%%%%%%%%%%%%%%%%%%%%%%%%%%%%%%%%%%%%%%%%%%%%%%
%																 INTRO
%%%%%%%%%%%%%%%%%%%%%%%%%%%%%%%%%%%%%%%%%%%%%%%%%%%%%%%%%%%%%%%%%%%%%%%%%%%%%%%%%%%%%%%%%
%%%%%%%%%%%%%%%%%%%%%%%%%%%%%%%%%%%%%%%%%%%%%%%%%%%%%%%%%%%%%%%%%%%%%%%%%%%%%%%%%%%%%%%%%
\section{Introduction}
The {\em generalized Ricci flow}, {\em gRF} for short, is a coupled geometric flow evolving a one-parameter family of Riemannian metrics $g_t$ and closed 3-forms $H_t$ 
on an oriented smooth manifold $M$ as follows 
\begin{equation}\label{auxiliar}
\begin{dcases}
\frac{\partial}{\partial t}g_t = -2\, \mathrm{Rc}_{g_t}+\frac12 H_t^2,\\
\frac{\partial}{\partial t}H_t = \Delta_{g_t}H_t. 
\end{dcases}
\end{equation}
In the system above, $\Delta_g = -(dd^*_g+d^*_gd)$ is the negative of the Hodge Laplacian of the Riemannian metric $g$, 
and the symmetric 2-tensor $H^2$ induced by the 3-form $H$ is given by $H^2(X,Y) = g(X\lrcorner H,Y\lrcorner H)$, for all $X,Y\in C^\infty(M,TM)$. 

\smallskip
The gRF was introduced in theoretical physics in the context of renormalization group flows of two-dimensional nonlinear sigma models \cite{1,2}. It can also be   
interpreted as a generalization of the Ricci flow to connections with torsion \cite{Streets}, and as a geometric flow of generalized metrics on exact Courant algebroids \cite{Streets2}. 
Recently, the gRF has been related to some geometric flows in Hermitian Geometry, like e.g.~the {\em pluriclosed flow} and the {\em generalized K\"ahler Ricci-flow} 
\cite{libro,Streets3,StTi}, and it has been studied on nilpotent Lie groups \cite{Par}. 
We refer the reader to \cite{libro} for an extensive introduction to this topic.    

\smallskip
In the compact case, the gRF is well-posed and it enjoys various properties that are akin to well-known properties of the Ricci flow (see e.g.~\cite{libro, 2} and compare with \cite{Ham,KlLo,Per}). 
In particular, it is the gradient flow of the lowest eigenvalue $\lambda(g,H)$ of the Schr\"odinger operator 
\[
\Phi_{g,H} \coloneqq -4\Delta_g+R_g-\frac{1}{12}|H|_g^2, 
\]
with respect to a suitable $L^2$ inner product on $C^{\infty}(M,S^2_{\sst+})\times C^{\infty}(M,\Lambda^3)$ (cf.~\cite[Prop.~3.4]{2}).  
This eigenvalue is characterized by the condition 
\begin{equation}\label{lambdafunct}
\lambda(g,H) = \inf_{\left\{f\in C^\infty(M) ~|~ \int_M{\rm e}^{-f}\,dV_g=1\right\}}\mathcal{F}(g,H,f), 
\end{equation}
where the {\em energy functional} $\mathcal{F}$ is defined as follows for any Riemannian metric $g$, closed 3-form $H$, and smooth function $f$ on $M$
\[
\mathcal{F}(g,H,f) \coloneqq \int_M \left(R_g-\frac{1}{12}|H|^2_g+|df|^2_g\right)\,{\rm e}^{-f}\,dV_g\,. 
\]
Moreover, the eigenfunction of $\Phi_{g,H}$ corresponding to $\lambda(g,H)$ is given by $\mathrm{e}^{-f_{g,H}/2}$, 
where $f_{g,H}$ is the unique minimizer of $\mathcal{F}$ under the constraint $\int_M\mathrm{e}^{-f}dV_g=1$. Equally, $f_{g,H}$ solves the equation
\begin{equation}\label{fgHeqn}
2\Delta_g f_{g,H} - |d f_{g,H}|_g^2 + R_g -\frac{1}{12}|H|^2_g = \lambda(g,H). 
\end{equation}

\smallskip
In \cite{BuHa}, the authors obtained dynamical stability and instability results for Ricci-flat metrics under the Ricci flow, improving the results previously obtained in \cite{T,Sesum}. 
In detail, they showed that if a Ricci-flat metric $\hat{g}$ is a local maximizer of Perelman's lambda-functional, then the solution of the Ricci flow starting close to it exists for all times 
and converges to a nearby Ricci-flat metric modulo diffeomorphisms (dynamical stability). 
In the case when $\hat{g}$ is not a local maximizer, they proved the existence of a non-trivial ancient solution $g_t$ of the Ricci flow 
that converges to $\hat{g}$ as $t\rightarrow-\infty$ (dynamical instability). 
 
Motivated by these results, in the present paper we investigate the dynamical behaviour of the gRF on compact manifolds. 
Instead of focusing on the functional $\lambda$ defined in \eqref{lambdafunct}, we fix a background closed 3-form $\widehat{H}$ and we consider the functional 
\[
\mu\colon C^\infty(M,S^2_{\sst+}) \times C^\infty(M,\Lambda^2) \to \R,\quad \mu(g,b) \coloneqq \lambda(g,\widehat{H}+db).  
\]
The critical points of $\mu$ and the stationary points of the gRF \eqref{auxiliar} constitute different sets. 
However, if we choose $\widehat{H}=0$, then the pairs of the form $(\hat{g},0)$, with $\hat{g}$ a Ricci-flat metric and $0\in C^\infty(M,\Lambda^2)$, are critical points of $\mu$ 
and give rise to stationary points of the gRF. The linear stability of such pairs under the gRF was investigated in \cite{2}.  

In \cite{StTi1}, Streets and Tian introduced and studied a family of geometric flows for Hermitian metrics, obtaining in particular a dynamical stability result for K\"ahler Ricci-flat metrics under all these flows. 
This applies to the pluriclosed flow \cite{StTi0,StTi, StTi2}, which belongs to the family of Hermitian curvature flows and it is an instance of gRF for a suitable choice of the closed 3-form $H$ 
(cf.~\cite[Thm.~6.5]{StTi2}).

\smallskip
Adapting the proof of \cite[Thm.~1]{BuHa} to this setting, we show the following dynamical stability result for the gRF.
\begin{theorem}\label{main1}
Let $(M,\hat{g})$ be a compact Ricci-flat manifold.  
Assume that $(\hat g,0)\in C^\infty(M,S^2_{\sst+}) \times C^\infty(M,\Lambda^2)$ is a local maximizer of $\mu(g,b) = \lambda(g,db)$. 
Then, there exists an open neighbourhood $\mathcal{U}$  of $(\hat{g},0)$ in $C^{\infty}(M,S^2_{\sst+})\times dC^{\infty}(M,\Lambda^2)$ in the $C^{\infty}$-topology 
such that the generalized Ricci flow starting at any $(g_0,db_0)\in \mathcal{U}$ has a long time solution which converges modulo diffeomorphisms to $(g_{\infty},0)$, 
with $g_{\infty}$ Ricci-flat. 
\end{theorem}

The proof of Theorem \ref{main1} is discussed in Section \ref{ProofSect}, where we also show some preliminary results including a Lojasiewicz-Simon inequality for $\mu$ (Lemma \ref{LSineq}). 

\smallskip 
As for the dynamical instability, we observe that if $(\hat g,0)$ is not a local maximizer of $\mu$,  
then it is not dynamically stable under the gRF.  
This can be easily deduced from \cite[Thm.~2]{BuHa}. Indeed, in such a case there exists a sequence $\{(g_i,db_i)\}$ in 
$C^{\infty}(M,S^2_{\sst+})\times dC^{\infty}(M,\Lambda^2)$ converging to $(\hat g,0)$ and such that  $\lambda(g_i,db_i) \geq \lambda(\hat g,0)$. 
Since $\lambda(g_i,0)\geq \lambda(g_i,db_i)$,  $\{g_i\}$ is a sequence in $C^{\infty}(M,S^2_{\sst+})$ converging to $\hat{g}$ and such that 
$\lambda(g_i,0)\geq \lambda(\hat g,0)$. 
Therefore, the metric $\hat{g}$ is not a local maximizer of Perelman's lambda-functional and by \cite[Thm.~2]{BuHa} there exists an ancient solution $g_t$ of the Ricci flow, 
with $t\in(-\infty,0]$, that converges modulo diffeomorphisms to $\hat{g}$ as $t\to -\infty$. 
Consequently, the pair $(g_t,0)$ is a an ancient solution of the gRF that converges modulo diffeomorphisms to $(\hat g,0)$ as $t\to -\infty$. 

\medskip
\noindent {\bf Notation.}
Given a vector bundle $\pi\colon E\to M$ over $M$, we denote by $C^\infty(M,E)$ the set of smooth sections of $E$ and by $C^\infty(M\times I,E)$ 
the set of smooth sections of $E$ depending on a real parameter $t\in I\subseteq \R$. We use similar notations for the spaces of $C^{k,\alpha}$ sections and $W^{q,p}$ sections.  
We use the shorthand $S^2_{\sst+}$ to denote the bundle $S^2_{\sst+}T^*M$ of positive definite symmetric 2-tensors on $M,$  
and $\Lambda^k$ to denote the bundle $\Lambda^kT^*M$ of exterior $k$-forms on $M.$

%%%%%%%%%%%%%%%%%%%%%%%%%%%%%%%%%%%%%%%%%%%%%%%%%%%%%%%%%%%%%%%%%%%%%%%%%%%%%%%%%%%%%%%%%
%%%%%%%%%%%%%%%%%%%%%%%%%%%%%%%%%%%%%%%%%%%%%%%%%%%%%%%%%%%%%%%%%%%%%%%%%%%%%%%%%%%%%%%%%
%																 FUNCTIONAL
%%%%%%%%%%%%%%%%%%%%%%%%%%%%%%%%%%%%%%%%%%%%%%%%%%%%%%%%%%%%%%%%%%%%%%%%%%%%%%%%%%%%%%%%%
%%%%%%%%%%%%%%%%%%%%%%%%%%%%%%%%%%%%%%%%%%%%%%%%%%%%%%%%%%%%%%%%%%%%%%%%%%%%%%%%%%%%%%%%%
\section{The functional $\mu$}
Let $M$ be a compact oriented smooth manifold and fix a background closed 3-form $\widehat{H}\in C^{\infty}(M,\Lambda^3)$. 
In this section, we review some useful properties of the functional  
\[
\mu\colon C^\infty(M,S^2_{\sst+}) \times C^\infty(M,\Lambda^2) \to \R,\quad \mu(g,b) \coloneqq \lambda(g,\widehat{H}+db). 
\]
The reader may refer to \cite{libro,2} for further details. We begin computing the gradient of $\mu$.
\begin{prop}\label{prop1} 
The gradient of $\mu$ at $(g,b)\in C^\infty(M,S^2_{\sst+}) \times C^\infty(M,\Lambda^2)$ with respect to the $L^2(M,{\rm e}^{-f_{g,H}}dV_g)$ inner product is given by  
\[
\nabla \mu (g,b)=\left(-{\rm Rc}_g-{\rm Hess}_{g}f_{g,H}+\frac{1}{4} H^2,-\frac{1}{2}d^*_{g}H-\frac{1}{2}\nabla f_{g,H}\lrcorner H \right),
\]
where $H=\widehat{H}+db$. 
\end{prop}
\begin{proof}
We follow the approach used in \cite{T} to study the variational structure of Perelman's lambda-functional.  
Let $\varepsilon$ be a real parameter, choose $(h,\beta)\in C^\infty(M,S^2) \times C^\infty(M,\Lambda^2)$, and consider the variations
\[
g_\varepsilon=g+\varepsilon h\,,\quad  b_{\varepsilon}=b+\varepsilon \beta\,,\quad H_\varepsilon=\widehat{H}+db_\varepsilon\,,\quad  \Phi_{\varepsilon}=\Phi_{g_\varepsilon,H_\varepsilon}. 
\]
Notice that $\mu(g,b) $ is the smallest eigenvalue of the Schr\"odinger operator $\Phi_0$ with corresponding normalized eigenvector $w_{g,b}\coloneqq{\rm e}^{-f_{g,H}/2}$, 
where $H \coloneqq H_ 0 = \widehat{H}+db$.  
Thus, $\mu(g_\varepsilon,b_\varepsilon)$ depends analytically on $\varepsilon$ \cite{RS}, and for $\varepsilon$ small enough we can consider the $L^2(M,dV_g)$-orthogonal projection $P_\varepsilon$ 
onto the one-dimensional $\mu(g_\varepsilon,b_\varepsilon)$-eigenspace of $\Phi_\varepsilon$. 
In this way 
\[
\Phi_\varepsilon P_\varepsilon w_{g,b}=\mu(g_\varepsilon,b_\varepsilon)P_\varepsilon w_{g,b}, 
\] 
and 
\[
\mu(g_\varepsilon,b_\varepsilon)=\mu(g)+\frac{\langle w_{g,b},(\Phi_\varepsilon-\Phi_0)P_\varepsilon w_{g,b}\rangle_{L^2(M,dV_g)}}{\langle w_{g,b},P_\varepsilon w_{g,b}\rangle_{L^{2}(M,dV_g)}}. 
\]
Consequently, we have
\[
\dez \mu(g_\varepsilon,b_\varepsilon) = \langle w_{g,b}, \left.\tfrac{d}{d\varepsilon}\right|_{\varepsilon=0} \Phi_\varepsilon\,w_{g,b} \rangle_{L^2(M,dV_g)}\,. 
\]
Now, $\Phi_\varepsilon = -4 \Delta_{g_\varepsilon}+R_{g_\varepsilon}-\tfrac{1}{12}|H_\varepsilon|_{g_\varepsilon}^2$. 
Using the variational formulae of the Laplacian operator and the scalar curvature, one has (cf.~\cite{T})
\[
\langle w_{g,b}, \dez(-4\Delta_{g_\varepsilon}+R_{g_\varepsilon}) w_{g,b}\rangle_{L^2(M,dV_g)}
= \int_M g(h, -{\rm Rc}_g - {\rm Hess}_g f_{g,H} )\,{\rm e}^{-f_{g,H}} dV_g\,.
\]
Moreover, by \cite[Lemma 5.3]{libro} the following identity holds
\[
\begin{split}
-\tfrac{1}{12} \langle w_{g,b}, \dez |H_{\varepsilon}|_{g_\varepsilon}^2 w_{g,b} \rangle_{L^2(M,dV_g)} 	=	&~\tfrac14 \langle w_{g,b}, g(h,H^2)w_{g,b} \rangle_{L^2(M,dV_g)}\\
																				 	& -\tfrac{1}{6}\langle w_{g,b}, g(d\beta,H)w_{g,b}\rangle_{L^2(M,dV_g)}\,.
\end{split}
\]
Finally, by \cite[(6.9)]{libro} the second summand in the RHS of the previous identity can be rewritten as follows 
\[
-\tfrac{1}{6}\langle w_{g,b}, g(d\beta,H)w_{g,b} \rangle_{L^2(M,dV_g)}= \int_M  g(\beta, -\tfrac12(d_g^* H +\nabla f_{g,H} \lrcorner H))\, {\rm e}^{-f_{g,H}} dV_g.
\]
Hence
\[
\begin{split}
\dez \mu(g_\varepsilon,b_\varepsilon) =	& \left\langle{-\rm Rc}_g - {\rm Hess}_g f_{g,H}+\tfrac14 H^2,h \right\rangle_{L^2(M,{\rm e}^{-f_{g,H}}  dV_g)}\\
								&+ \langle -\tfrac12(d_g^*H+\nabla f_{g,H} \lrcorner H) ,\beta \rangle_{L^2(M, {\rm e}^{-f_{g,H}}  V_g)}, 
\end{split}
\]
and the statement follows.  
\end{proof}

Recall that Ricci-flat metrics are stationary points of the Ricci flow and they are the critical points of Perelman's lambda-functional \cite{Per}. 
Moreover, this functional vanishes at any scalar-flat metric and thus at any Ricci-flat metric. 
The situation for the functional $\mu$ we are considering is slightly different. 

The stationary points of the gRF are given by pairs $(g,H)\in C^\infty(M,S^2_{\sst+})\times C^\infty(M,\Lambda^3)$  satisfying the conditions
\begin{equation}\label{fixedgRF} 
\mathrm{Rc}_g -\tfrac14 H^2 = 0,\quad \Delta_gH = 0, 
\end{equation}
with $dH=0$. 
On the other hand, if we fix a background closed 3-form $\widehat{H}$ and we let $H \coloneqq \widehat{H} + db$, then the critical points of the functional $\mu(g,b) = \lambda(g,\widehat{H}+db)$ satisfy
\begin{equation}\label{criticallambda}
\mathrm{Rc}_g + \mathrm{Hess}_{g}f_{g,H} - \tfrac{1}{4} H^2 = 0 ,\quad  d^*_{g}H + \nabla f_{g,H}\lrcorner H = 0, 
\end{equation}
with $f_{g,H}$ solving the equation \eqref{fgHeqn}. 

Notice that $\mu(g,b)=\lambda(g,\widehat{H}+db)$ vanishes whenever  
\begin{equation}\label{GenScalFlat}
R_{{g}} - \tfrac{1}{12}|H|_{{g}}^2=0.
\end{equation} 
Indeed, in this case $\Phi_{{g},H} = -4\Delta_{{g}}$, whence it follows that $\mu({g},b)=0$ and $f_{{g},H}$ is constant.  
The condition \eqref{GenScalFlat} is not satisfied by the stationary points of the gRF, as one can see tracing the first equation in \eqref{fixedgRF}. 
Similarly, one can show that a critical point of $\mu$ may not be a stationary point of the gRF.

\smallskip
The previous observations lead us considering pairs $(\hat{g},0)\in C^{\infty}(M,S^2_{\sst+})\times C^{\infty}(M,\Lambda^2)$, with $\hat{g}$ a Ricci-flat metric. 
Indeed, if we choose $\widehat{H}=0$, then $(\hat{g},0)$ is a stationary point of the gRF and a critical point of $\mu$, $\mu(\hat{g},0)$ is zero and  $f_{\hat{g},0}$ is constant. 
On the other hand, we have the following. 

\begin{lemma}\label{grad0}
Let $(g,b)\in C^{\infty}(M,S^2_{\sst+})\times C^{\infty}(M,\Lambda^2)$ be a critical point of $\mu$ and assume that $\mu(g,b) = \lambda(g,\widehat{H}+db)\leq 0$.  
Then, $g$ is Ricci-flat and $\widehat{H} + db = 0$. 
\end{lemma} 
\begin{proof}
Let $H\coloneqq \widehat{H} + db$, $f \coloneqq f_{g,H}$ and $\mu\coloneqq \mu(g,b)$. 
Tracing the first equation in \eqref{criticallambda}, we obtain
\[
R_g + \Delta_g f - \frac14|H|_g^2 = 0.
\]
Then, as $f$ solves \eqref{fgHeqn}, we have
\[
\Delta_g f - |d f |_g^2 +\frac{1}{6}|H|^2_g = \mu. 
\]
Multiplying both sides of the previous identity by $\mathrm{e}^{-f}$ and integrating over $M$ gives
\[
\int_M\Delta_g \mathrm{e}^{-f}dV_g -\frac16 \int_M |H|_g^2\, \mathrm{e}^{-f} dV_g = - \mu \int_M \mathrm{e}^{-f}dV_g, 
\]
whence
\[
\frac16 \int_M |H|_g^2\, \mathrm{e}^{-f} dV_g =  \mu. 
\]
Therefore, the assumption $\mu\leq 0$ implies that $\mu=0$ and $\widehat{H}+db = 0$.  In particular, $f$ is constant and the first equation of \eqref{criticallambda} gives $\Rc_g=0$. 
\end{proof}

\begin{rem}
Notice that the hypothesis $\mu(g,b)\leq0$ in Lemma \ref{grad0} is necessary, due to the existence of non-trivial  (i.e., with non-constant $f_{g,H}$) solitons for the gRF \cite{Streets4,StUs}. 
\end{rem}

We now determine the linearization of $\nabla\mu$ at $(\hat{g},0)$.
\begin{prop}\label{hessian}
Let $\hat g$ be a Ricci-flat metric and $\widehat{H}=0$. Then, the linearization of $\nabla\mu$ at $(\hat g,0)$ is given by
\[
L(h,\beta)=-\frac{1}{2} \left(\Delta_{\hat g}^{L}h,d^*_{\hat g}d\beta  \right), 
\]
for every $(h,\beta)\in C^\infty(M,S^2) \times C^\infty(M,\Lambda^2)$ such that ${\rm div}_{\hat g}h=0$, where $\Delta^L_{\hat g}$ is the Lichnerowicz Laplacian of $\hat{g}$.   
\end{prop}
\begin{proof}
Let $g_\varepsilon \coloneqq \hat{g} + \varepsilon h$, with $h \in C^\infty(M,S^2)$, and let $f_\varepsilon \coloneqq f_{g_{\varepsilon},\varepsilon d\beta }$. 
From the proof of Proposition \ref{prop1}, we have 
\[
\dez \nabla \mu(\hat g+\varepsilon h,\varepsilon \beta) =  
\dez \left(-{\rm Rc}_{g_\varepsilon}-{\rm Hess}_{g_\varepsilon} f_{\varepsilon}
+\tfrac{1}{4} \varepsilon^2 (d\beta)^2,-\tfrac{1}{2}\varepsilon d^*_{{g}_{\varepsilon}}d\beta -\tfrac{1}{2}\varepsilon \nabla f_{\varepsilon}\lrcorner d\beta \right). 
\]
Taking into account that $f_0=f_{\hat{g},0}$ is constant, and using the variational formulae of the Ricci tensor and the Hessian operator, we obtain  
\[
\dez \nabla \mu(\hat g+\varepsilon h,\varepsilon \beta) =
\begin{dcases} 
\left(-\frac12 \Delta_{\hat g}^{L}h,-\frac{1}{2} d^*_{\hat g}d\beta  \right), \mbox{ if }\, {\rm div}_{\hat g}h=0,\\
\left(0,-\frac{1}{2} d^*_{\hat g}d\beta  \right), \quad \mbox{ if }\, h\in (\ker {\rm div}_{\hat g})^{\perp}.
\end{dcases}
\] 
\end{proof}

%%%%%%%%%%%%%%%%%%%%%%%%%%%%%%%%%%%%%%%%%%%%%%%%%%%%%%%%%%%%%%%%%%%%%%%%%%%%%%%%%%%%%%%%%
%%%%%%%%%%%%%%%%%%%%%%%%%%%%%%%%%%%%%%%%%%%%%%%%%%%%%%%%%%%%%%%%%%%%%%%%%%%%%%%%%%%%%%%%%
%																 THEOREM PROOF
%%%%%%%%%%%%%%%%%%%%%%%%%%%%%%%%%%%%%%%%%%%%%%%%%%%%%%%%%%%%%%%%%%%%%%%%%%%%%%%%%%%%%%%%%
%%%%%%%%%%%%%%%%%%%%%%%%%%%%%%%%%%%%%%%%%%%%%%%%%%%%%%%%%%%%%%%%%%%%%%%%%%%%%%%%%%%%%%%%%
\section{Proof of Theorem \ref{main1}}\label{ProofSect}
Before proving Theorem \ref{main1}, we show two preliminary lemmas.  
The first one is a Lojasiewicz-Simon inequality for $\mu$, which is obtained applying \cite[Thm.~6.3]{colding} in the same spirit of \cite{BuHa}. 
The second lemma involves a gauge fixing of the gRF and it is obtained applying the Nash-Moser inverse function theorem \cite{HamiltonNash} in the same fashion as in \cite{BryantXu,Ham}.  
\begin{lemma}\label{LSineq}
Let $(M,\hat g)$ be a compact Ricci-flat manifold and let $\widehat{H}=0$.
Then, there exist $\varepsilon>0$ and $\theta\in \left(0,\tfrac12\right]$ such that
\begin{equation}\label{rotula}
\left\|\left({\rm Rc}_g+{\rm Hess}_{g}f_{g,H} - \tfrac{1}{4} H^2,\tfrac{1}{2}\left(d^*_{g}H+
	\nabla f_{g,H}\lrcorner H\right) \right)\right\|_{L^2(M,{\rm e}^{-f_{g,H}}dV_g)} \geq |\lambda(g,H)|^{1-\theta}, 
\end{equation}
for every $g\in C^{2,\alpha}(M,S^{2}_{\sst+})$ and $H=db$, with $b\in C^{2,\alpha}(M,\Lambda^2)$, such that $\|(g-\hat g,b)\|_{C^{2,\alpha}}<\varepsilon$.
\end{lemma}
\begin{proof}
By the Ebin slice theorem \cite{Ebin}, we can find an open neighbourhood $\mathcal{V}$ of $\hat{g}$ in $C^{2,\alpha}(M,S_{\sst+}^2)$  
and a $\sigma>0$ such that every metric $g\in\mathcal{V}$ can be written as $g=\f^*(\hat{g}+h)$, for some $\f\in\mathrm{Dif{}f}(M)$ and some $h\in  C^{2,\alpha}(M,S^2)$ satisfying
${\rm div}_{\hat g}h=0$ and $\|h\|_{C^{2,\alpha}}<\sigma$.

Since \eqref{rotula} is ${\rm Dif{}f}(M)$-invariant and it only involves $H=db$, we may assume that $(g,b)$ belongs to the space  
\[
\mathcal{S} \coloneqq \left\{\hat g+h\in C^{2,\alpha}(M,S^2_{\sst+}) ~|~ {\rm div}_{\hat g} h = 0,~\|h\|_{C^{2,\alpha}}<\sigma\right\} \times 
\left\{\beta\in C^{2,\alpha}(M,\Lambda^2) ~|~ d^{*}_{\hat g}\beta=0\right\}. 
\]
Let us consider the restriction of $\mu$ to $\mathcal{S}$
\[
\mu_\mathcal{S}\colon \mathcal{S}\to \R\,,\quad \mu_\mathcal{S}(g,b)=\lambda(g,db). 
\]
We claim that $\mu_\mathcal{S}$ satisfies the hypothesis of Colding-Minicozzi theorem \cite[Thm.~6.3]{colding}.  
From standard perturbation theory \cite{RS}, it follows that $\mu_\mathcal{S}$ is analytic. 
By Proposition \ref{prop1}, the $L^2(M,{\rm e}^{-f_{g,H}} dV_g)$-gradient of $\mu_\mathcal{S}$ at $(g,b)$ is   
\[
\nabla \mu_{\mathcal{S}}(g,b) = \pi \left(-{\rm Rc}_g-{\rm Hess}_{g}f_{g,H}+\tfrac{1}{4} H^2,-\tfrac{1}{2}(d^*_{g}H-\nabla f_{g,H}\lrcorner H) \right),
\]
where $H=db$ and $\pi$ is the $L^2(M,{\rm e}^{-f_{g,H}} dV_g)$-orthogonal projection onto the tangent space  
\[
T_{(\hat g,0)}\mathcal S=\{h\in C^{2,\alpha}(M,S^2) ~|~ {\rm div}_{\hat g} h = 0\} \oplus \{\beta\in C^{2,\alpha} ~|~ d^{*}_{\hat g}\beta=0\}. 
\]
From standard elliptic theory it follows that
\[
\begin{split}
\|\nabla\mu_\mathcal{S}(g_1,b_1)-\nabla\mu_\mathcal{S}(g_2,b_2)\|_{C^{0,\alpha}} 	&\leq C\|(g_1-g_2,b_1-b_2)\|_{C^{2,\alpha}},\\
\|\nabla\mu_\mathcal{S}(g_1,b_1)-\nabla\mu_\mathcal{S}(g_2,b_2)\|_{L^2}			&\leq C\|(g_1-g_2,b_1-b_2)\|_{W^{2,2}},
\end{split}
\]
for every $(g_1,b_1)$ and $(g_2,b_2)$ in $\mathcal{S}$.  
From Proposition \ref{hessian}, the linearization of $\nabla \mu_\mathcal{S}$ at $(\hat{g},0)$ is 
\[
L_\mathcal{S}=-\frac12\left(\Delta_{\hat g}^{L}\,,d^*_{\hat g}d\,  \right). 
\]
Our claim then follows, and by \cite[Thm.~6.3]{colding} there exists $\theta \in \left(0,\tfrac12\right]$ such that 
\[
\left\|\nabla \mu_\mathcal{S}(g,b)\right\|_{L^2(M,{\rm e}^{-f_{g,H}}dV_g)}\geq  |\mu_\mathcal{S}(g,b)|^{1-\theta}\,. 
\]
for every $(g,b)\in \mathcal{S}$. This last formula implies \eqref{rotula}. 
\end{proof} 

\smallskip
\begin{lemma}\label{DeTurk}
Let $(\bar{g},\bar{H})\in C^{\infty}(M,{S}^2_{\sst+}) \times C^{\infty}(M,\Lambda^3)$ be a stationary point of the gRF. 
For every $T>0,~\varepsilon>0$, there exists $\delta>0$ such that if $(g_0,H_0)\in C^{\infty}(M,{S}^2_{\sst+}) \times C^{\infty}(M,\Lambda^3)$, with $dH_0=0$, satisfies  
\[
  \|(g_0-\bar g,H_0- \bar{H})\|_{C^{\infty}} < \delta,
\]
then the solution $(g_t,H_t)$ of the gRF starting at $(g_0,H_0)$ is defined for $t\in [0,T')$, with $T'>T$, 
and there exists a smooth family of diffeomorphisms $\{\varphi_t\}_{t\in [0,T')}$ such that $\varphi_0={\rm Id}$ and 
\[
\|(\varphi_t^*(g_t) -\bar{g},\varphi_t^*(H_t) -\bar{H}\|_{C^{\infty}}<\varepsilon.  
\] 
\end{lemma}
\begin{proof}
Let $g$ be a Riemannian metric on $M.$ Following the proof of the short-time existence of the gRF \cite[Sect.~5.2]{libro}, we denote by $X_g$ the vector field 
\[
X_g={\rm tr}_g(D_g-D_{\bar{g}}),
\]
where $D_g$ is the Levi-Civita connection of $g$.  

We let
\[
\mathscr{F} \coloneqq C^{\infty}(M\times [0,T],S^2_{\sst+})\times 
C^{\infty}(M\times [0,T],\Lambda^3) 
\]
and 
\[
\mathscr{G} \coloneqq \mathscr{F}\times C^{\infty}(M,S^2_{\sst+})\times
C^{\infty}(M,\Lambda^3) \,.
\]
Both $\mathscr{F}$ and $\mathscr{G}$ are tame Fr\'echet spaces with respect to the gradings  
\[
\|(g, H)\|_{n}	= \,\sum_{2j\leq n}\int_{0}^T\|(\partial_t^{j} g_t,\partial_t^{j} H_t)\|_{W^{n-2j,2}}\,dt,
\]
and 
\[
\|(g, H),(g', H')\|_n =\, \|(g, H)\|_{n} + \|(g', H')\|_{W^{n,2}},
\]
respectively. 
Consider the map $F\colon \mathscr{F} \to \mathscr{G}$ defined by 
\[
F(g,H) = \left( \left(\partial_tg + 2\,\Rc_{g} - \tfrac12 H^2 - \mathcal{L}_{X_g}g, \partial_tH + dd^*_{g}H - d (X_g\lrcorner H)\right), 
(g|_{t=0},H|_{t=0})\right),
\]
where $\mathcal L$ denotes the Lie derivative.
Note that if $H$ is closed we have 
\[
F(g,H) = \left( \left(\partial_tg + 2\,\Rc_{g} - \tfrac12 H^2 - \mathcal{L}_{X_g}g, \partial_tH - \Delta_{g}H - \mathcal{L}_{X_g} H\right), 
(g|_{t=0},H|_{t=0})\right).
\]
From the results of \cite[Sect.~5.2]{libro}, it follows that  
\[
\left. F \right|_{*(g,H)}(h,K) =
 \left(\partial_t h-\Delta^L_{g} h+\Phi(h,K),\partial_tK-\Delta_{g} K+d\Psi(h,K),\left(h|_{t=0},K|_{t=0}\right)\right),
\]
for every closed 3-form $K$, where $\Phi(h,K)$ is first order in $h,K$ and $\Psi$ is first order in $h$ and zeroth order in $K$. 
Hence, the assumptions of the Hamilton-Nash-Moser Theorem \cite[Thm.~5.1]{Ham} with integrability condition $dK=0$ are satisfied, 
and from the proof of  \cite[Thm.~5.1]{Ham} it follows that  $F |_{*(g,H)}$ is an isomorphism for every $(g,H)\in C^{\infty}(M,S^2_{\sst+})\times C^{\infty}(M,\Lambda^3)$.

Let $(\bar{g}_t, \bar{H}_t) \equiv (\bar{g},\bar H)\in\mathscr{F}$. Since 
\[
F(\bar{g},\bar H) = \left(\left(0,0\right),\left( \bar{g},\bar H\right)\right), 
\]
the map $F$ is invertible from an open neighbourhood of $(\bar{g},\bar H)$ in $\mathscr{F}$ to an open neighbourhood of $\left(\left(0,0\right),\left(\bar{g},\bar H\right)\right)$ in $\mathscr{G}$. 
Thus, for every $\varepsilon>0$ there exists $\delta>0$ such that if $(\tilde g_0,\tilde H_0)\in C^{\infty}(M,S^2_{\sst+}) \times C^{\infty}(M,\Lambda^3)$ satisfies 
$\|(\tilde g_0- \bar{g}, \tilde H_0 - \bar{H})\|_{C^{\infty}}<\delta$, then the initial value problem  
\[
\begin{dcases}
\tfrac{\partial}{\partial t}\tilde g_t = -2\,\Rc_{ \tilde g_t}+\tfrac12 \tilde H_t^2 +\mathcal{L}_{X_{\tilde{g}_t}}\tilde{g}_t, \\   \
\tfrac{\partial}{\partial t}\tilde H_t = -dd^*_{\tilde g_t}\tilde H_t + d (X_{\tilde{g}_t}\lrcorner \tilde{H}_t),\\  
\left. \tilde{g}_t \right|_{t=0} = \tilde{g}_0,\\
\tilde{H}_t|_{t=0} = \tilde{H}_0,
\end{dcases}
\]
has a unique solution $(\tilde g_t,\tilde H_t)$ defined for $t\in [0,T')$, with $T'>T$, and such that $\|(\tilde g_t - \bar{g}, \tilde H_t - \bar{H})\|_{C^\infty}<\varepsilon$. 
If we further assume that $d\tilde H_0=0$, then $\tilde{H}_t$ stays closed for every $t\in [0,T')$, and $(\tilde{g}_t,\tilde{H}_t)$   
solves the generalized DeTurck-Ricci flow 
\[
\begin{dcases}
\tfrac{\partial}{\partial t}\tilde g_t = -2\,\Rc_{ \tilde g_t}+\tfrac12 \tilde H_t^2+\mathcal{L}_{X_{\tilde g_t}} \tilde g_t,\\
\tfrac{\partial}{\partial t}\tilde H_t = \Delta_{\tilde g_t}\tilde H_t+\mathcal{L}_{X_{\tilde g_t}}\tilde H_t,\\
d\tilde H_t=0, \\
\left.  \tilde{g}_t \right|_{t=0} = \tilde g_0,\\
\tilde{H} |_{t=0} = \tilde H_0.
\end{dcases}
\]
The thesis follows by choosing $\{\psi_t\}\in{\rm Dif{}f}(M)$ solving
\[
\partial_t\psi_t = - X_{g_t}\circ \psi_t\,,\quad \psi_0={\rm Id},  
\] 
so that $(g_t,H_t) = (\psi_{t}^*(\tilde{g}_t),\psi_{t}^*(\tilde{H}_t))$ is a solution of the gRF. 
\end{proof} 

\smallskip
We are now ready to prove Theorem \ref{main1}.

\begin{proof}[Proof of Theorem $\ref{main1}$] 
Let $B_{r}$ denote the ball of radius $r$ and center $(\hat g,0)$ in $C^{\infty}(M,S^2_{\sst+})\times dC^{\infty}(M,\Lambda^2)$, where the $C^k$-norms are defined using $\hat{g}$.   
Since $\mu(\hat{g},0)=0$ and $(\hat{g},0)$ is a local maximizer of $\mu$, in view of Lemma \ref{LSineq} we choose $\varepsilon>0$ so that for every $(g,db)\in B_{\varepsilon}$ we have 
$\mu(g,b) = \lambda(g,db)\leq 0$ and
\[
\|\left({\rm Rc}_g+{\rm Hess}_{g}f_{g,H}-\tfrac{1}{4} H^2,\tfrac{1}{2}d^*_{g}H+\tfrac{1}{2}\nabla f_{g,H}\lrcorner H \right)\|_{L^2(M,{\rm e}^{-f_{g,H}}dV_g)}\geq |\lambda(g,H)|^{1-\theta}\,,\\  
\]
for some $\theta\in\left(0,\tfrac12\right]$, with $H=db$.

By Lemma \ref{DeTurk}, there exists $0<\delta<\varepsilon$ small enough so that for every $(g_0,H_0)\in B_{\delta}$ the gRF starting at $(g_0,H_0)$ has a solution $(g_t,H_t)$ 
defined for $t\in [0,T')$, with $T'>1$,  
and there exists a smooth family of diffeomorphisms $\{\varphi_t\}$ such that $\varphi_{0}={\rm Id}$ and $(\varphi_t^*g_t,\varphi_t^*H_t)\in B_{\varepsilon/4}$, for every $t\in [0,T')$. 

Fix $(g_0,H_0)\in B_\delta$, and let $T>0$ be the maximal time such that for every $t\in [0,T)$ there exists $\varphi_t\in {\rm Dif{}f}(M)$ such that $(\varphi_{t}^*g_t,\varphi_t^*H_t)\in B_{\varepsilon}$. 
Our choice of $\delta$ implies that $T \geq 1$ and that there exists $\varphi_1 \in {\rm Dif{}f}(M)$ such that $(\varphi_1^* g_1, \varphi_1^* H_1)\in B_{\varepsilon/4}$.

Let $X_t=-\nabla f_{t}$, where $f_t \coloneqq f_{\varphi_1^* g_t, \varphi_1^* H_t}$ and 
the gradient is taken with respect to $\varphi_1^* g_t$, and let $\{\psi_t \}_{t\in [0,T)}\subseteq {\rm Dif{}f}(M)$ be the family of diffeomorphisms generated by $X_t$ and satisfying $\psi_1={\rm Id}.$   
For $t\in [0,T)$, let $\tilde{\f}_t = \varphi_1 \circ \psi_t$ and let $\tilde g_t=\tilde{\varphi}_t^*g_t$, $\tilde H_t = \tilde{\varphi}_t^*H_t$.  Then 
\[
\partial_t\tilde g_t=-2({\rm Rc}_{\tilde g_t}+{\rm Hess}_{\tilde g_t}\tilde f_{t})+ \tfrac{1}{2}\tilde H^2_t,\quad 
\partial_t\tilde H_t = \Delta_{\tilde{g_t}}\tilde H_t-d\left(\nabla\tilde f_t\lrcorner \tilde H_t\right), 
\]
where $\tilde{f}_t = f_{\tilde{\varphi}_t^*g_t,\tilde{\varphi}_t^*H_t}$.

Now, we can write $\tilde H_t = d \tilde b_t$, with $\tilde{b}_t$ solving
\[
\partial_t\tilde b_t=     - d^*_{\tilde g_t}d \tilde{b}_t   - \nabla \tilde{f}_{t}\lrcorner d\tilde{b}_t. 
\]
By interpolation we then have 
\begin{equation}\label{interpolalo}
\|(\partial_t \tilde g_t,\partial_t\tilde H_t)\|_{C^{k}}= \|(\partial_t \tilde g_t,d\partial_t \tilde b_t)\|_{C^{k}}\leq 
\|(\partial_t \tilde g_t,\partial_t \tilde b_t)\|_{C^{k+1}}\leq C\|(\partial_t \tilde g_t,\partial_t \tilde b_t)\|_{L^2}^{1-\eta},
\end{equation}
for some $\eta\in(0,1)$. 
We let 
\[
T'' \coloneqq \sup\{ t\in [1,T] ~|~ (\tilde g_t,\tilde H_t)\in B_\varepsilon\}, 
\]
and $\sigma \coloneqq \theta-\eta+\theta\eta>0$. 
Then, by Lemma \ref{rotula} and the inequality \eqref{interpolalo}, we obtain 
\[
\begin{split}
-\frac{d}{dt}|\mu(\tilde g_t,\tilde b_t)|^{\sigma}	&=\,\sigma |\mu(\tilde g_t,\tilde b_t)|^{\sigma-1} \frac{d}{dt} \mu(\tilde g_t,\tilde b_t)\\
									&=\,2\sigma |\mu(\tilde g_t,\tilde b_t)|^{(\theta-1)(1+\eta)} \left\|\left({\rm Rc}_{\tilde g_t}+{\rm Hess}_{\tilde g_t}\tilde f_{t}-\tfrac{1}{4} \tilde H^2_t,
									\tfrac{1}{2}d^*_{\tilde g_t}\tilde H_t+\tfrac{1}{2}\nabla\tilde f_{t}\lrcorner \tilde H \right)\right\|_{L^2}^{1+\eta}\|(\partial_t\tilde g_t,\partial_t\tilde b_t)\|^{1-\eta}_{L^2}\\
									&\geq \frac{\sigma}{C}\|(\partial_t\tilde g_t,\partial_t\tilde b_t)\|_{C^{k}}\,.
\end{split}
\]
Therefore, for every integer $k$ sufficiently large we have
\[
\begin{aligned}
\int_1^{T''}\|(\partial_t\tilde g_t,\partial_t\tilde H_t)\|_{C^{k-1}}\,dt & \leq \int_1^{T''}\|(\partial_t\tilde g_t,\partial_t\tilde b_t)\|_{C^{k}}\,dt\leq \frac{C}{\sigma}|\mu(\tilde g_1,\tilde b_1)|^{\sigma}& \\
&\leq  \frac{C}{\sigma}|\mu( \tilde{g}_0, \tilde{b}_0)|^{\sigma}	=\frac{C}{\sigma}|\lambda(g_0, db_0)|^{\sigma},
\end{aligned}
\]
where the last inequality follows since 
\[
\partial_t\tilde g_t = -{\rm Rc}_{\tilde g_t} -{\rm Hess}_{\tilde g_t}\tilde f_{t}+ \tfrac{1}{4}\tilde H^2_t,\quad  \partial_t\tilde b_t=     -\tfrac12 d^*_{\tilde g_t}d \tilde{b}_t  -\tfrac12 \nabla \tilde{f}_{t}\lrcorner d\tilde{b}_t , 
\] 
is the gradient flow of $\mu$ (cf.~Proposition \ref{prop1}), while the last equality follows from the diffeomorphism invariance of $\mu$.

Now, up to shrinking $\delta$, we may assume that $\frac{C}{\sigma}|\lambda(g_0, H_0)|^{\sigma}\leq \tfrac{\varepsilon}{4}$. 
This implies that $T''=\infty$, as otherwise we would have $(\tilde g_{T''},\tilde H_{T''})\in B_{\varepsilon/2}$, which is a contradiction. 
Hence, $(\tilde g_t, d\tilde b_t)$ is defined for every positive $t$, and it converges in the 
$C^{\infty}$-topology to a pair $(\tilde g_{\infty},d\tilde{b}_{\infty})\in B_\varepsilon$. 
In particular,  $\mu(\tilde g_{\infty},\tilde{b}_{\infty}) = \lambda(\tilde g_{\infty},d\tilde{b}_{\infty}) \leq 0$. 
Since $(\partial_t\tilde g_t,\partial_t\tilde b_t)\to 0$ in the $C^\infty$-topology, we also have $\nabla\mu(\tilde g_{\infty},\tilde b_{\infty})=0$.
We can then apply Lemma \ref{grad0} to conclude that $ d\tilde b_{\infty}=0$ and $\Rc_{\tilde g_{\infty}}=0$. 
\end{proof}

\bigskip\noindent
{\bf Acknowledgements.} The authors would like to thank Reto Buzano, Mario Garcia Fern\'andez, Fabio Paradiso and Jeffrey Streets for useful comments and conversations.  
The authors were supported by GNSAGA of INdAM. 
A.R.~was also supported by the project PRIN 2017  ``Real and Complex Manifolds: Topology, Geometry and Holomorphic Dynamics''.

\end{document}